\documentclass[A4paper,reqno,oneside,12pt]{amsart}
\usepackage{fullpage}
\usepackage{setspace}
\usepackage{datetime}
\usepackage{tikz}
\usetikzlibrary{er}
\usetikzlibrary{arrows.meta}
\usepackage{amsmath}
\usepackage{amsfonts}
\usepackage{amssymb}
\usepackage{verbatim,enumitem,graphics}
\usepackage{microtype}
\usepackage{xcolor}
\usepackage[backref=page]{hyperref}
\hypersetup{%
    colorlinks,
    linkcolor={red!50!black},
    citecolor={green!50!black},
    urlcolor={blue!50!black}
}
\usepackage{dsfont}

\usepackage[utf8]{inputenc}

\usepackage[abbrev,msc-links,backrefs]{amsrefs} 
\usepackage{doi}

\renewcommand{\PrintDOI}[1]{\doi{#1}}

\usepackage{fnpct}

\usepackage{MnSymbol}

\setenumerate{leftmargin=*} 

\usepackage{enumitem}

\newtheorem{theorem}{Theorem}[section]
\newtheorem{lemma}[theorem]{Lemma}
\newtheorem{proposition}[theorem]{Proposition}

\newtheorem{corollary}[theorem]{Corollary}

\newtheorem{definition}[theorem]{Definition}

\def\({\left(}
\def\){\right)}  

\newcommand{\oldqed}{}
\def\endofClaim{\hfill\scalebox{.6}{$\Box$}}

\newcommand{\ZZ}{\mathbb{Z}}

\newcommand{\PP}{\mathbb{P}}
\newcommand{\RR}{\mathbb{R}}

\newcommand{\EE}{\mathbb{E}}

\let\epsilon\varepsilon
\let\phi\varphi

\DeclareMathOperator{\AP}{AP}
\DeclareMathOperator{\diam}{diam}
\DeclareMathOperator{\lcm}{lcm}

\usepackage{lineno}
\newcommand*\patchAmsMathEnvironmentForLineno[1]{%
\expandafter\let\csname old#1\expandafter\endcsname\csname #1\endcsname
\expandafter\let\csname oldend#1\expandafter\endcsname\csname end#1\endcsname
\renewenvironment{#1}%
{\linenomath\csname old#1\endcsname}%
{\csname oldend#1\endcsname\endlinenomath}}%
\newcommand*\patchBothAmsMathEnvironmentsForLineno[1]{%
\patchAmsMathEnvironmentForLineno{#1}%
\patchAmsMathEnvironmentForLineno{#1*}}%
\AtBeginDocument{%
\patchBothAmsMathEnvironmentsForLineno{equation}%
\patchBothAmsMathEnvironmentsForLineno{align}%
\patchBothAmsMathEnvironmentsForLineno{flalign}%
\patchBothAmsMathEnvironmentsForLineno{alignat}%
\patchBothAmsMathEnvironmentsForLineno{gather}%
\patchBothAmsMathEnvironmentsForLineno{multline}%
}

\begin{document}
\shortdate
\yyyymmdddate
\settimeformat{ampmtime}
\footskip=28pt
\allowdisplaybreaks

\title{A blurred view of Van der Waerden type theorems}

\author{Vojtech R\"{o}dl}

\author{Marcelo Sales}

\thanks{The first author was supported by NSF grant DMS 1764385.\\ The second author was partially supported by NSF grant DMS 1764385}

\address{Department of Mathematics, Emory University, 
    Atlanta, GA, USA}
\email{\{vrodl|mtsales\}@emory.edu}

\begin{abstract}
 Let $\AP_k=\{a,a+d,\ldots,a+(k-1)d\}$ be an arithmetic progression. For $\epsilon>0$ we call a set $\AP_k(\epsilon)=\{x_0,\ldots,x_{k-1}\}$ an $\epsilon$-approximate arithmetic progression if for some $a$ and~$d$,~$|x_i-(a+id)|<\epsilon d$ holds for all $i\in\{0,1\ldots,k-1\}$. Complementing earlier results of Dumitrescu \cite{D11}, in this paper we study numerical aspects of Van der Waerden, Szemer\'edi and Furstenberg--Katznelson like results in which arithmetic progressions and their higher dimensional extensions are replaced by their $\epsilon$-approximation.
\end{abstract}

\dedicatory{Dedicated to the memory of Ronald Graham}

\maketitle

\section{Introduction}\label{sec:intro}

For a natural number $N$ we set $[N]=\{1,2\ldots,N\}$. Assume that $[N]$ is colored by $r$ colors. We denote by 
\begin{align*}
    N \rightarrow (\AP_k)_r
\end{align*}
the fact that any such $r$-coloring yields a monochromatic arithmetic progression $\AP_k$ of length $k$. With this notation the well known Van der Waerden's theorem can be stated as follows.

\begin{theorem}\label{th:Vanderwaerden}
For every positive integers $r$ and $k$, there exists a positive integer $N$ such that~$N~\rightarrow~(\AP_k)_r$.
\end{theorem}

The minimum $N$ with the property of Theorem \ref{th:Vanderwaerden} is called the Van der Waerden number of $r, k$ and is denoted by $W(k,r)$. In other words, $W(k,r)$ is the minimum integer $N$ such that any $r$-coloring of $[N]$ contains a monochromatic arithmetic progression of length $k$. Much effort was put to determine lower and upper bounds for $W(k,r)$, but the problem remains widely open. As an illustration, the best known bounds for $W(k,2)$ are
\begin{align*}
    \frac{2^k}{k^{o(1)}}\leq W(k,2) \leq 2^{2^{2^{2^{2^{k+9}}}}},
\end{align*}
where $o(1)\rightarrow 0$ as $k\rightarrow \infty$. The lower bound is due to Szabo \cite{Sz90} while the upper bound is a celebrated result of Gowers on Szemer\'edi's theorem \cite{G01}. It is good to remark that when $k$ is a prime the lower bound can be improved to $W(k+1,2)\geq k2^k$ by a construction of Berlekamp~\cite{B68}.

Ron Graham was keenly interested in the research leading to improvements of the upper bound of $W(k,2)$ and motivated it by monetary prizes. Currently open is his $\$1000$ award for the proof that $W(k,2)<2^{k^2}$ (see \cite{G08}). During his career he also contributed to related problems in the area (see \cite{BGL99, Gn86, RG06}). For instance, together with Erd\H{o}s \cite{EG79}, Graham proved a canonical version of Van der Waerden: Every coloring of $\mathbb{N}$, not necessarily with finitely many colors, contains either an monochromatic arithmetic progression or a rainbow arithmetic progression, i.e., a progression with every element of distinct color.

Inspired by the works of \cite{D11} and \cite{HKSS19}, we are interested in the related problem where we replace an arithmetic progression by an perturbation of it.

\begin{definition}\label{def:ap_k(epsilon)}
Given $\epsilon>0$, a set $X=\{x_0,\ldots,x_{k-1}\}\subseteq [N]$ is an $\epsilon$-approximate $\AP_k(\epsilon)$ of an arithmetic progression of length $k$ if there exists $a \in \RR$ and $d>0$ such that $|x_i-(a+id)|<\epsilon d$.
\end{definition}

In other words, an $\AP_k(\epsilon)$ is just a transversal of $\bigcup_{i=0}^{k-1}B(a+id; \epsilon d)$, where $B(a+id; \epsilon d)$ is the open ball centered at $a+id$ of radius $\epsilon d$. Depending on the choice of $\epsilon$, an $\AP_k(\epsilon)$ can be different from an $\AP_k$. For example, if $\epsilon=1/3$, then $a=0.8$ and $d=2.4$ testifies that~$\{1,3,6\}$ is an $\epsilon$-approximate arithmetic progression of length 3, but it is not an arithmetic progression itself. 

For integers $r$, $k$ and $\epsilon>0$, let 
\begin{align*}
    W_{\epsilon}(k,r)=\min\{N:\: N\rightarrow (\AP_k(\epsilon))_r\}.
\end{align*}
That is, $W_{\epsilon}(k,r)$ is the smallest $N$ with the property that any coloring of $[N]$ by $r$ colors yields a monochromatic $\AP_k(\epsilon)$. Our first result shows that one can obtain sharper bounds to the Van der Waerden problem by replacing $\AP_k$ to $\AP_k(\epsilon)$.

\begin{theorem}\label{th:approximatewaerden}
Let $r\geq 1$. There exists a positive constant $\epsilon_0$ and a real number $c_r$ depending on $r$ such that the following holds. If $0<\epsilon\leq \epsilon_0$ and $k\geq 2^rr!\epsilon^{-1}\log^r(1/5\epsilon)$, then
\begin{align*}
c_r\frac{k^r}{\epsilon^{r-1}\log(1/\epsilon)^{\binom{r+1}{2}-1}}\leq W_{\epsilon}(k,r) \leq \frac{2k^r}{\epsilon^{r-1}}.
\end{align*}
\end{theorem}

Similar as in the previous discussion we will write $N\rightarrow_{\alpha} \AP_k$ (or $N\rightarrow_{\alpha} \AP_k(\epsilon)$) to denote that any subset $S\subseteq [N]$ with $|S|\geq \alpha N$ necessarily contains an arithmetic progression $\AP_k$ (or $\AP_k(\epsilon)$, respectively). Answering a question of Erd\H{o}s and Turan \cite{ET36}, Szemer\'edi proved the following celebrated result:

\begin{theorem}\label{th:szemeredi}
For any $\alpha>0$ and a positive integer $k$, there exists an integer $N_0$ such that for every $N\geq N_0$ the relation $N\rightarrow_{\alpha} \AP_k$ holds.
\end{theorem}

Basically Szemer\'edi theorem states that any positive proportion of $\mathbb{N}$ contains an arithmetic progression of length $k$. Not much later Furstenberg \cite{F77} gave an alternative proof of Theorem~\ref{th:szemeredi} using Ergodic theory. Extending \cite{F77}, Furstenberg and Katznelson \cite{FK78} were able to prove a multidimensional version of Szemer\'edi's theorem:

An $m$-dimensional cube $C(m,k)$ is a set of $k^m$ points in $m$-dimensional Euclidean lattice~$\ZZ^m$ such that
\begin{align*}
    C(m,k)=\{\vec{a}+d\vec{v}:\: \vec{a}=(a_1,\ldots,a_m) \in \ZZ^m \text{ and } \vec{v}=(v_1,\ldots,v_m) \in \{0,1,\ldots,k-1\}^m\}.
\end{align*}
That is, $C(m,k)$ is a homothetic translation of $[k]^m$. As in the one dimensional case, for~$\alpha>0$ and integers $m$, $k$ and $N$ we will write $[N]^m \rightarrow_{\alpha} C(m,k)$ to mean that any subset~$S \subseteq [N]^m$ with $|S|\geq \alpha N^m$ contains a cube $C(m,k)$. The following is the multidimensional version of Theorem \ref{th:szemeredi} proved in \cite{FK78}.

\begin{theorem}\label{th:furstenberg}
For any $\alpha>0$ and positive integers $k$ and $m$, there exists an integer $N_0$ such that for every $N\geq N_0$ the relation $[N]^m \rightarrow_{\alpha} C(m,k)$ holds
\end{theorem}

Define $f(N,m,k)$ as the maximum size of a subset $A \subseteq [N]^m$ without a cube $C(m,k)$. Note that $f(N,1,k)$ corresponds to the maximal size of a subset $A\subseteq [N]$ without an arithmetic progression $\AP_k$. Theorems \ref{th:szemeredi} and $\ref{th:furstenberg}$ give us that $f(N,m,k)=o(N^m)$. Determining bounds for $f(N,m,k)$ is a long standing problem in additive combinatorics. For $m=1$ the best current bounds are
\begin{align*}
N \exp\left(-c_k(\log N)^{1/\lceil\log_2 k \rceil}\right)    \leq f(N,1,k)\leq\frac{N}{(\log \log N)^{2^{-2^{k+9}}}}
\end{align*}
where $c_k$ is a positive constant depending only on $k$. The upper bound is due to Gowers \cite{G01}, while the lower bound with best constant $c_k$ is due to O'Bryant \cite{OB11}.

For larger $m$ it is worth mentioning that Furstenberg--Katznelson proof of Theorem \ref{th:furstenberg} uses Ergodic theory and gives us no quantitative bounds on $f(N,m,K)$. Purely combinatorial proofs were given later based on the hypergraph regularity lemma in \cite{G07} and \cite{RS04, NRS06}. Those proofs give quantitative bounds which are incomparably weaker than the one for $m=1$. For instance, in \cite{MS19} Moshkovitz and Shapira proved that the hypergraph regularity lemma gives a bound of the order of the $k$-th Ackermann function.

Now we consider $\epsilon$-approximate versions of Theorems \ref{th:szemeredi} and \ref{th:furstenberg}. 

\begin{definition}\label{def:cube}
Given $\epsilon>0$, a set $X=\{x_{\vec{v}}:\: \vec{v} \in \{0,1,\ldots,k-1\}^m\}\subseteq [N]^m$ is an $\epsilon$-approximate cube $C_{\epsilon}(m,k)$ if there exists $\vec{a} \in \RR^m$ and $d>0$ such that $||x_{\vec{v}}-(\vec{a}+d\vec{v})||<\epsilon d.$ 
\end{definition}

For integers $N$, $m$, $k$ and $\epsilon>0$, let $f_{\epsilon}(N,m,k)$ be the maximal size of a subset~$A~\subseteq~[N]^m$ without an $C_{\epsilon}(m,k)$. Dimitrescu showed an upper bound for $f_{\epsilon}(N,m,k)$ in \cite{D11}. We complement his result by also providing a lower bound to the problem.

\begin{theorem}\label{th:approximateszemeredi}
Let $m\geq 1$ and $k\geq 3$ be integers and $0<\epsilon<1/125$. Then there exists an integer $N_0:=N_0(k,\epsilon)$ and positive constants $c_1$ and $c_2$ depending only on $k$ and $m$ such that
\begin{align*}
    N^{m-c_1(\log(1/\epsilon))^{\frac{1}{\ell}-1}}\leq f_{\epsilon}(N,m,k)\leq N^{m-c_2(\log(1/\epsilon))^{-1}},
\end{align*}
for $N\geq N_0$ and $\ell=\lceil \log_2 k \rceil$.
\end{theorem}

The paper is organized as follows. In Section \ref{pf:waerden}, we present a proof of Theorem \ref{th:approximatewaerden}. The upper bound is an iterated blow-up construction, while the lower bound is given by an ad-hoc inductive coloring. We prove Theorem \ref{th:approximateszemeredi} in Section \ref{pf:szemeredi}. The lower bound uses the current lower bounds for $f(N,1,k)$, while the upper bound is given by an iterated blow-up construction combined with an averaging argument.

\section{Proof of Theorem \ref{th:approximatewaerden}}\label{pf:waerden}

\subsection{Upper bound} 

We start with the upper bound. Given $r\geq 1$ colors, we consider the following $r$-iterated blow-up of an $\AP_k$ given by the set of integers 
\begin{align*}
    B_r=\left\{b_0+tb_1+\ldots+t^{r-1}b_{r-1}:\: (b_0,\ldots,b_{r-1})\in \{0,1\ldots,k-1\}^r,\, t=\lceil k/\epsilon \rceil \right\}
\end{align*}
Note that $B_r$ is a set of size $|B_r|=k^r$ and $\diam(B_r)\leq (k-1)(1+t+\ldots+t^{r-1})<2(k-1)t^{r-1}$. It turns out that any $r$-coloring of $B_r$ contains a monochromatic $\AP_k(\epsilon)$. In particular, this implies that $W_{\epsilon}(k,r)\leq \diam(B_r)+1\leq 2k^r/\epsilon^{r-1}$.

\begin{proposition}\label{prop:upperboundwaerden}
Any $r$-coloring of $B_r$ has a monochromatic $\AP_k(\epsilon)$.
\end{proposition}

\begin{proof}
We prove the proposition by induction on the number of colors $r$. For $r=1$, one can see that $B_1=[k]$, which is an arithmetic progression of length $k$ and in particular a $\AP_k(\epsilon)$. Now suppose that any $(r-1)$-coloring of $B_{r-1}$ contains a monochromatic $\AP_k(\epsilon)$. Consider an $r$-coloring of $B_r$. Note that we can partition $B_r=\bigcup_{i=0}^{k-1} B_{r,i}$ where
\begin{align*}
    B_{r,i}=\left\{b_0+\ldots+t^{r-2}b_{r-2}+it^{r-1}:\: (b_0,\ldots,b_{r-2})\in \{0,1,\ldots,k-1\}^{r-1},\, t=\lceil k/\epsilon \rceil \right\}.
\end{align*}
That is, for every $0\leq i \leq k-1$, the set $B_{r,i}$ is a translation of $B_{r-1}$ by $it^{r-1}$. 

Consider a transversal $X=\{x_0,\ldots,x_{k-1}\}$ of $B_r=\bigcup_{i=0}^{k-1} B_{r,i}$ with $x_i \in B_{r,i}$ for every $0\leq i \leq k-1$. Let $a=\diam(B_{r-1})/2$ and $d=t^{r-1}$. Since $x_i \in B_{r,i}$ implies that~$it^{r-1}\leq x_i \leq~it^{r-1}+~\diam(B_{r-1})$, we obtain that
\begin{align*}
    |x_i-(a+id)|\leq \frac{\diam(B_{r-1})}{2} \leq \frac{k^{r-1}}{\epsilon^{r-2}} \leq \epsilon d
\end{align*}
and $X$ is an $\epsilon$-approximate $\AP_k(\epsilon)$. Therefore, if some color $c$ is present in each of the sets~$B_{r,i}$ for $0\leq i \leq r-1$, we could select $X$ to be a monochromatic $\AP_k(\epsilon)$. Consequently we may assume that there is no monochromatic transversal in $B_r$, which means that there exists an index $i$ such that $B_{r,i}$ is colored with at most $(r-1)$ colors. Since $B_{r,i}$ is just a translation of $B_{r-1}$, by induction hypothesis we conclude that there exists a monochromatic~$\AP_k(\epsilon)$ inside $B_{r,i}$.
\end{proof}

\subsection{Lower bound}

In order to construct a large set avoiding $\epsilon$-approximate $\AP_k(\epsilon)$ we need some preliminary results. Given a real number $D>0$, we define an \textbf{$(r-1,1;D)$-alternate labeling of $\RR$} to be an labeling $\chi: \RR \rightarrow \{-1,+1\}$ such that

\begin{align*}
\chi(x)=\begin{cases}
+1&, \quad \text{if } x\in \bigcup_{i\in\ZZ}\left(irD+mD,\left(i+\frac{r-1}{r}\right)rD+mD\right],\\
-1&, \quad \text{if } x\in\bigcup_{i\in \ZZ}\left(\left(i+\frac{r-1}{r}\right)rD+mD, (i+1)rD+mD\right],
\end{cases}
\end{align*}
for some $m \in \ZZ$. That is, $\chi$ is a periodic labeling of $\RR$ with period $rD$, where we partition $\RR$ into disjoint intervals of length $D$ and label them alternating between $r-1$ consecutive intervals of label $+1$ and one of label $-1$. The restriction of an $(r-1,1;D)$-alternate labeling to $\ZZ$ will be of great importance for us. The following lemma roughly characterizes the common difference of any large monochromatic approximate arithmetic progression in such a labeling.

\begin{lemma}\label{lem:water}
Let $D,\delta>0$, $m$ be a positive integer with $\delta\leq \frac{1}{2r(r+1)}$ and $\chi: \RR \rightarrow \{-1,+1\}$ be an $(r-1,1;D)$-alternate labeling of $\RR$. If there exist $a, d\in \RR$ and an integer $\ell$ such that
\begin{align*}
    d\not \in \bigcup_{i\in \ZZ}\bigcup_{q=1}^{r}\left(\left(\frac{i}{q}-\delta\right)rD, \left(\frac{i}{q}+\delta\right)rD\right),
\end{align*}
and that $B=\bigcup_{i=0}^{\ell-1}B(a+id,\delta r D)$ has a monochromatic transversal of label $+1$, then $\ell\leq 3r/\delta$.
\end{lemma}

\begin{proof}
We may assume without loss of generality that $\chi$ is the following labeling of $\RR$:
\begin{align*}
    \chi(x)=\begin{cases}
+1&, \quad \text{if } x\in \bigcup_{i\in\ZZ}\left(irD,\left(i+\frac{r-1}{r}\right)rD\right],\\
-1&, \quad \text{if } x\in\bigcup_{i\in \ZZ}\left(\left(i+\frac{r-1}{r}\right)rD, (i+1)rD\right],
\end{cases}
\end{align*}
That is, we may assume that $m=0$ in the definition of an alternate labeling. Also, during the proof we shall write $\overline{x}$ to be the representative of $x$ modulo $rD$ in the interval $(0,rD]$, i.e., the number $0<\overline{x}\leq rD$ such that $x-\overline{x}=brD$ for some integer $b\in \ZZ$. 

We start by claiming that there exists $1\leq s \leq r$ such that
\begin{align}\label{eq:1}
    \overline{sd} \in \left[\delta rD, \frac{rD}{r+1}\right] \cup \left[\left(1-\frac{1}{r+1}\right)rD, (1-\delta)rD\right].
\end{align}
First note by our hypothesis that 
\begin{align*}
    d\notin \left(\left(\frac{i}{q}-\delta\right)rD, \left(\frac{i}{q}+\delta\right)rD\right) 
\end{align*}
for every $i \in \ZZ$ and $1\leq q \leq r$. Therefore,
\begin{align}\label{eq:2}
    qd \notin ((i-\delta)rD, (i+\delta)rD) \subseteq ((i-q\delta)rD, (i+q\delta)rD)
\end{align}
for every $i \in \ZZ$ and $1\leq q \leq r$.

Now consider the partition $(0,rD]=\bigcup_{j=0}^r\left(\frac{jrD}{r+1},\frac{(j+1)rD}{r+1}\right]$. If there exists $1\leq s \leq r$ such that~$\overline{sd}$ is in the two outer intervals above, i.e., in either $\left(0,\frac{rD}{r+1}\right]$ or $\left(\left(1-\frac{1}{r+1}\right)rD,rD\right]$, then by (\ref{eq:2}) we obtain that $s$ satisfies (\ref{eq:1}). Otherwise, assume that there is no $1\leq s \leq r$ with $\overline{sd}$ in the two outer intervals. Then by the pigeonhole principle there exist~$1\leq p<q \leq r$ and an index $j$ such that $\overline{pd}, \overline{qd}\in \left(\frac{jrD}{r+1},\frac{(j+1)rD}{r+1}\right]$. Consequently, we have that $\overline{qd}-\overline{pd}\in~\left(-\frac{rD}{r+1}, \frac{rD}{r+1}\right)$. By letting $s=q-p$ we obtain that 
\begin{align*}
    \overline{sd} \in \left(0,\frac{rD}{r+1}\right] \cup \left(\left(1-\frac{1}{r+1}\right)rD,rD\right],
\end{align*}
for $1\leq s \leq r$, which is a contradiction. Therefore, condition (\ref{eq:1}) is always satisfied for some $s$.

Let $1\leq s \leq r$ be the number satisfying (\ref{eq:1}) and consider the subset
\begin{align*}
    B'=\bigcup_{i=0}^{\ell'}B(a+isd,\delta r D) \subseteq  B,
\end{align*}
where $\ell'=\lfloor(\ell-1)/s \rfloor$.
That is, if we see $B$ as the arithmetic progression of intervals of length~$\delta rD$, size $\ell$ and common difference $d$, then $B'$ is a subarithmetic progression of $B$ with common difference $sd$. Since $B$ has a monochromatic transversal labeled $+1$, then $B'$ also has a monochromatic transversal labeled $+1$. Hence, because $\bigcup_{i\in \ZZ}\left(irD,(i+\frac{r-1}{r})rD\right]$ are the elements of label $+1$ in our $(r-1,1;D)$-alternate labeling, we have that 
\begin{align*}
    \{a,a+sd,\ldots,a+\ell' sd\} \subseteq \bigcup_{i \in \ZZ} \left((i-\delta)rD, \left(i+\frac{r-1}{r}+\delta\right) rD\right).
\end{align*}

Suppose that $\overline{sd} \in \left[\delta rD,\frac{1}{r+1}rD\right]$. Since the coloring $\chi$ is periodic modulo $rD$, we may assume without loss of generality that $sd \in \left[\delta rD,\frac{rD}{r+1}\right]$. We claim that there exists an integer $p$ such that $\{a,a+sd,\ldots,a+\ell' sd\} \subseteq  \left((p-\delta)rD, \left(p+\frac{r-1}{r}+\delta\right) rD\right)$. Suppose that this is not the case. Because $sd>0$ there exist integers $p<q$ and $0\leq i\leq \ell'-1$ such that~$a+isd \in \left((p-\delta)rD, \left(p+\frac{r-1}{r}+\delta\right) rD\right)$ and $a+(i+1)sd \in \left((q-\delta)rD, \left(q+\frac{r-1}{r}+\delta\right) rD\right)$. A computation shows that
\begin{align*}
    sd=a+(i+1)sd-(a+isd)> (q-\delta)rD-\left(p+\frac{r-1}{r}+\delta\right)rD\geq (1-2\delta r)D \geq \frac{rD}{r+1}
\end{align*}
for $\delta\leq \frac{1}{2r(r+1)}$, which contradicts our assumption on $sd$. 

Hence, there exists $p$ such that $a, a+\ell'sd \in \left((p-\delta)rD, \left(p+\frac{r-1}{r}+\delta\right) rD\right)$, which implies that
\begin{align*}
    \ell' sd= (a+\ell' sd)-a\leq \left(p+\frac{r-1}{r}+\delta\right) rD-(p-\delta)rD=(r-1)D+2\delta r D.
\end{align*}
Since $sd \geq \delta r D$, we obtain that
\begin{align*}
    \ell'sd \geq \left\lfloor\frac{\ell-1}{s}\right\rfloor\delta r D \geq \frac{\ell\delta r D}{2s} \geq \frac{\delta \ell D}{2}
\end{align*}
for $\ell>r\geq s$. The last two computations combined with the fact that $\delta\leq \frac{1}{2r(r+1)} \leq \frac{1}{4}$ gives us that
\begin{align*}
    \ell\leq \frac{2(r-1)D+4\delta rD}{\delta  D}\leq \frac{2(r-1)}{\delta}+4r\leq\left(\frac{2}{\delta}+4\right)r \leq \frac{3r}{\delta}
\end{align*}

Now assume that $\overline{sd} \in  \left[\left(1-\frac{1}{r+1}\right)rD,(1-\delta)rD\right]$. By the periodicity of $\chi$, we may assume without loss of generality that $sd \in [-\frac{rD}{r+1},-\delta rD]$. By rewriting $\{a,a+sd,\ldots,a+\ell' sd\}$ as~$\{a',a'+sd',\ldots,a'+\ell' sd'\}$ with $a'=a+\ell' sd$ and $d'=-d$, we are back to the previous case and again $\ell\leq 3r/\delta$.
\end{proof}

Although it is convenient to prove Lemma $\ref{lem:water}$ using an alternate labeling of $\RR$, the lower bound construction will use alternate labelings of set of integers. With this in mind, we give the following companion definition. 

Given positive integers $D$, $r$ and $t$, an \textbf{$(r-1,1;D)$-alternate labeling} of the set $[rtD]$ is a labeling $\chi':[rtD] \rightarrow \{-1,+1\}$ such that $\chi'(x)=\chi(x)$, where~$\chi$ is an $(r-1,1;D)$-alternate labeling of $\RR$. In other words, an alternate labeling of a set of integers is just the restriction of an alternate labeling of $\RR$ to the set. Note that by this definition, there exists~$r$ distinct~$(r-1,1;D)$-alternate labelings of $[rtD]$. A \textbf{$D$-block} of $[rtD]$ is a block of~$D$ consecutive integers of the form $[iD+1,(i+1)D]$. One can note that the $D$-blocks form a partition of $[rtD]$ and each $D$-block is monochromatic in an $(r-1,1;D)$-alternate labeling of $[rtD]$.  

Finally, note that given an alternate labeling $\chi'$ of a set $[rtD]$ we can extend back to a an alternate labeling of $(0,rtD]$ by labeling the entire interval $(iD,(i+1)D]$ with the same label as the $D$-block of integers $[iD+1,(i+1)D]$. Since the labeling is periodic, it is now easy to extend back to a labeling $\chi$ of $\RR$. 

The next result is a consequence of the proof of Lemma \ref{lem:water}.

\begin{proposition}\label{prop:simplelowerbound}
Let $D$, $r$, $t$ and $\ell$ be positive integers with $\ell\geq t(r+1)+2$ and $0<\epsilon<1/2r$ be a real number. If $[rtD]$ is colored by an $(r-1, 1; D)$-alternate labeling and $X\subseteq [rtD]$ is a monochromatic $\AP_{\ell}(\epsilon)$ of label $+1$, then there exists $0\leq i \leq rt-1$ such that the $D$-block~$[iD+1,(i+1)D]$ satisfies~$|X\cap [iD+1,(i+1)D]|\geq \ell/(r-1)$.
\end{proposition}

\begin{proof}
Write $X=\{x_0,\ldots,x_{\ell-1}\}$. Since $X$ is an $\AP_{\ell}(\epsilon)$, there exists $a\in \RR$, $d>0$ such that~$|x_i-(a+id)|<\epsilon d$. Therefore, a computation shows that
\begin{align*}
    rtD>|x_{\ell-1}-x_0|\geq a+(\ell-1)d-a-2\epsilon d=(\ell-1-2\epsilon)d,
\end{align*}
which implies that 
\begin{align}\label{eq:3}
    d\leq \frac{rtD}{\ell-2} \leq \frac{rD}{r+1}
\end{align}
for $\ell\geq t(r+1)+2$. 

Similarly as in the proof of Lemma \ref{lem:water}, we will show that all the elements of $X$ are inside an interval of $(r-1)$ consecutive $D$-blocks of label $+1$. 

Suppose that this was not the case. Since non-consecutive $D$-blocks of label $+1$ are at a distance of at least $D$ elements, then there exists $x_i$ and $x_{i+1}$ such that $|x_{i+1}-x_i|\geq D$. However, in view of $\epsilon<1/2r$ and (\ref{eq:3}), we obtain
\begin{align*}
    |x_{i+1}-x_i|\leq |x_{i+1}-(a+(i+1)d)|+|a+(i+1)d-(a+id)|+|x_{i}-(a+id)|\leq (1+2\epsilon)d < D,
\end{align*}
which is a contradiction. The result now follows by an application of the pigeonhole principle.
\end{proof}

Note that Proposition \ref{prop:simplelowerbound} already gives us a lower bound for the case $r=2$. Indeed, we will prove that an $(1,1;k-1)$-alternate labeling of $\left[\frac{2(k-1)(k-2)}{3}\right]$\footnote{Strictly speaking we should use the set $\left[2\left\lfloor\frac{k-2}{3} \right\rfloor(k-1)\right]$, since $\frac{k-2}{3}$ is not necessarily an integer. However, during our exposition we will not bother with this type of detail since it has no significant effect on arguments or results} does not contain a monochromatic~$\AP_k(\epsilon)$ for $\epsilon<1/4$ and sufficiently large $k$.

Suppose that this is not the case. Since an $(1,1;k-1)$-alternate labeling is symmetric, we may assume that there is a monochromatic $\AP_k(\epsilon)$ of label $+1$. Applying Proposition~\ref{prop:simplelowerbound} with $r=2$, $t=(k-2)/3$, $D=k-1$ and $\ell=k$ gives us that there exists a $(k-1)$-block of the form $[i(k-1)+1,(i+1)(k-1)]$ such that $|X\cap [i(k-1)+1,(i+1)(k-1)]|\geq k$, which contradicts the size of the block.

Unfortunately, the argument above does not give a lower bound depending on $\epsilon$. To achieve such a bound we will need to refine the previous construction, but first we need one more preliminary result. 

The second Chebyshev function $\psi(x)$ is defined to be the logarithm of the least common multiple of all positive integers less or equal than $x$. The following bound on $\psi(x)$ will be useful for us.

\begin{theorem}[\cite{RSch75}, Theorem 7]\label{th:psi}
If $x\geq 10^8$, then $|\psi(x)-x|<cx/\log x$ for some positive constant $c$.
\end{theorem}

In particular, Theorem \ref{th:psi} asserts that for sufficiently large $n$ we have
\begin{align}\label{eq:4}
    \lcm(1,\ldots,n)=e^{n+O(n/\log n)}.
\end{align}
We are now ready to prove the lower bound of Theorem \ref{th:approximatewaerden}.

\begin{theorem}\label{th:lowerboundwaerden}
Let $r\geq 1$. There exists a positive constant $\epsilon_0$ and a real number $c_r$ depending on $r$ such that the following holds. If $0<\epsilon\leq \epsilon_0$ and $k\geq 2^rr!\epsilon^{-1}\log^r(1/5\epsilon)$ is a integer, then there exist an integer~$N:=N(\epsilon,k,r)$ satisfying 
\begin{align*}
    N\geq c_r\frac{k^r}{\epsilon^{r-1} \log(1/\epsilon)^{\binom{r+1}{2}-1}},
\end{align*}
so that $[N]$ admits an $r$-coloring without monochromatic~$\AP_k(\epsilon)$.
\end{theorem} 

\begin{proof}
The proof is by induction on the number of colors $r$. For $r=1$, the result clearly holds for $N(\epsilon,k,1)=k-1$ since there is no $\AP_k(\epsilon)$, or even $\AP_k$, on $(k-1)$ terms. Now suppose that for any $\epsilon$ and $k$ such that $0<\epsilon\leq \epsilon_0$ and $k\geq 2^{r-1}(r-1)!\epsilon^{-1}\log^{r-1}(1/5\epsilon)$, there exists~$N(\epsilon,k,r-1)$ and a $(r-1)$-coloring of $[N(\epsilon,k,r-1)]$ satisfying the conclusion of the statement. We want to find an integer $N_1$ so that $[N_1]$ has a $r$-coloring without monochromatic $\AP_k(\epsilon)$.

To do that we start with some choice of variables. Let
\begin{align}\label{eq:5}
    N_0=N\left(\epsilon, \frac{k}{rs}, r-1\right), \quad s=\frac{1}{0.9}\log(1/5\epsilon), \quad w=\frac{e^{0.9s}}{s(r-1)!}, \quad t=\frac{k}{2rs}, \quad D_j=\frac{s-j+1}{s}N_0
\end{align}
be integers for $1\leq j \leq s/2$. Note that although $s, w, t$ and $\{D_j\}_{1\leq j\leq s/2}$ might not be integers, we prefer to write in this way, since it simplifies the exposition and has no significant effect on the arguments. Moreover, the integer $N_0$ always exists since by hypothesis
\begin{align*}
\frac{k}{rs}\geq \frac{2^rr!\epsilon^{-1}\log^r (1/5\epsilon)}{rs}\geq 2^{r-1}(r-1)!\epsilon^{-1}\log^{r-1}(1/5\epsilon).
\end{align*}

Let $N_1=rwt(D_1+\ldots+D_{s/2})$. We are going to define a coloring $\phi: [N_1]\rightarrow [r]$ not admitting monochromatic $\AP_k(\epsilon)$. To this end we partition $[N_1]$ into consecutive intervals following the four steps below:

\begin{itemize}
    \item First we partition $[N_1]$ into $[N_1]=Y_1\cup \ldots \cup Y_w$, where $Y_i$ are consecutive intervals and $|Y_i|=rt(D_1+\ldots+D_{s/2})$ for every $i=1,\ldots,w$.
    \item Each $Y_i$ is partitioned into $Y_i=Y_{i,1}\cup \ldots \cup Y_{i,s/2}$, where $Y_{i,j}$'s are consecutive intervals and $|Y_{i,j}|=rtD_j$ for every $j=1,\ldots,s/2$.
    \item Each $Y_{i,j}$ is partitioned into $Y_{i,j}=Z_{1}^{i,j}\cup\ldots \cup Z_{t}^{i,j}$, where $Z_{u}^{i,j}$'s are consecutive intervals and $|Z_{u}^{i,j}|=rD_j$ for every $u=1,\ldots,t$.
    \item Each $Z_u^{i,j}$ is partitioned into $Z_u^{i,j}=Z_{u,1}^{i,j}\cup \ldots \cup Z_{u,r}^{i,j}$, where $Z_{u,v}^{i,j}$'s are consecutive intervals and $|Z_{u,v}^{i,j}|=D_j$ for every $v=1,\ldots,r$.
\end{itemize}


More explicitly, we define
\begin{align*}
    \alpha_i&=(i-1)rt(D_1+\ldots+D_{s/2}), \quad i\in [w]\\
    \beta_{i,1}&=\alpha_i, \quad i\in[w]\\
    \beta_{i,j}&=rt(D_1+\ldots+D_{j-1})+\alpha_i, \quad (i,j)\in[w]\times[2,s/2]\\
    \gamma_{i,j,u}&=(u-1)rD_j+\beta_{i,j}, \quad (i,j,u)\in [w]\times[s/2]\times [t]\\
    \sigma_{i,j,u,v}&=(v-1)D_j+\gamma_{i,j,u}\quad (i,j,u,w)\in [w]\times[s/2]\times [t] \times [r].
\end{align*}
Therefore, our intervals can be written as
\begin{align*}
    Y_i&=[\alpha_i+1,\alpha_i+rt(D_1+\ldots+D_{s/2})], \quad i\in [w]\\
    Y_{i,j}&=[\beta_{i,j}+1,\beta_{i,j}+rtD_j], \quad (i,j)\in[w]\times [s/2]\\
    Z_{u}^{i,j}&=[\gamma_{i,j,u}+1,\gamma_{i,j,u}+rD_j], \quad (i,j,u)\in [w]\times[s/2]\times[t]\\
    Z_{u,v}^{i,j}&=[\sigma_{i,j,u,v}+1,\sigma_{i,j,u,v}+D_j], \quad (i,j,u,v)\in [w]\times[s/2]\times[t]\times [r]
\end{align*}

Finally, we describe the coloring $\phi:[N_1]\rightarrow [r]$ on the intervals $Z_{u,v}^{i,j}$. By induction hypothesis, given any set~$C$ of~$r-1$ colors there exists a coloring $\phi_C: [N_0]\rightarrow C$ with no monochromatic $\AP_{k/rs}(\epsilon)$. Fix~$Z_{u,v}^{i,j}$ with $(i,j,u,v)\in [w]\times [s/2] \times [t] \times [r]$. We color $Z_{u,v}^{i,j}$ by the same coloring as the first~$D_j$ elements of $[N_0]$ when $[N_0]$ is colored by $\phi_{[r]\setminus \{v\}}$. That is, the coloring $\phi$ restricted to $Z_{u,v}^{i,j}$ only uses $r-1$ colors and does not contain a monochromatic~$\AP_{k/rs}(\epsilon)$.

To prove that the coloring $\phi$ is free of $\AP_k(\epsilon)$ we are going to show that there is no~$a \in \RR$ and~$d>0$ such that $\bigcup_{i=0}^{k-1}B(a+id, \epsilon d)$ has a monochromatic transversal in $[N_1]$. Suppose the opposite and assume that there exists $a$ and $d$ such that $\bigcup_{i=0}^{k-1}B(a+id, \epsilon d)$ has a monochromatic transversal $X=\{x_0,\ldots,x_{k-1}\} \subseteq [N_1]$ of color $c\in [r]$. Since all the balls have radius $\epsilon d$, we obtain that $\{a,a+d,\ldots,a+(k-1)d\}\subseteq (1-\epsilon d, N_1+\epsilon d)$, which gives that~$(k-1)d\leq (N_1-1)+2\epsilon d$. By (\ref{eq:5}) and by the fact that $\epsilon\leq \epsilon_0$ we have that
\begin{align}\label{eq:6}
 d\leq \frac{N_1-1}{k-1-2\epsilon}\leq \frac{2N_1}{k}=\frac{2rwt(D_1+\ldots+D_{s/2})}{k}=\frac{w N_0}{s^2}\left(s+\ldots+\left(\frac{s}{2}+1\right)\right)\leq \frac{w N_0}{2},
\end{align}
for sufficiently small $\epsilon_0$.

For a fixed $Y_{i,j}=\bigcup_{u=1}^t\bigcup_{v=1}^rZ_{u,v}^{i,j}$ we define an auxiliary labeling $\chi_{i,j}:Y_{i,j} \rightarrow \{-1,+1\}$ of $Y_{i,j}$ such that every~$D_j$-block $Z_{u,v}^{i,j}$ is monochromatic and
\begin{align*}
\chi_{i,j}(Z_{u,v}^{i,j})=\begin{cases}
+1&, \quad \text{if } v\neq c,\\
-1&, \quad \text{if } v=c.
\end{cases}.
\end{align*}
In other words, every element of a $D_j$-block $Z_{u,v}^{i,j}$ is of label $-1$ if the coloring $\phi$ restricted to $Z_{u,v}^{i,j}$ has the same coloring of the first $D_j$ elements of $\phi_C:[N_0]\rightarrow C$, where $C=[r]\setminus\{c\}$, i.e., the set of colors missing the color $c$. Otherwise, we label all the elements in $Z_{u,v}^{i,j}$ by $+1$. It is not difficult to check that $\chi_{i,j}$ is an~$(r-1,1;D_j)$-alternate labeling of $Y_{i,j}$. Moreover, since $X$ is monochromatic of color~$c$ and $Z_{u,c}^{i,j}$ is colored by $\phi_{[r]\setminus\{c\}}$, we obtain that $X\cap Z_{u,c}^{i,j}=\emptyset$. This implies that every element of $X\cap Y_{i,j}$ is labeled $+1$. Finally, in order to apply Lemma \ref{lem:water}, we extend the labeling $\chi_{i,j}$ to the set of real numbers~$(\beta_{i,j},\beta_{i,j}+rtD_j]$ by labeling the entire interval $(\sigma_{i,j,u,v},\sigma_{i,j,u,v}+D_j]$ by color $\chi_{i,j}(Z_{u,v}^{i,j})$ for every ${u,v}\in [t]\times[r]$.

The main idea of the proof is based on the fact that for $d$ not too small, there exists an index $j_0$ such that~$d$ is far from certain fractions involving $D_{j_0}$. We will then imply by Lemma \ref{lem:water} that the number of elements of $X$ in $Y_{i,j_0}$ is ``small". It turns out that this fact is enough to restrict the entire location of~$X$ to just a few $Y_{i,j}$'s. Then by the pigeonhole principle and Proposition \ref{prop:simplelowerbound} we can show that there exists a $D_j$-block $Z_{u,v}^{i,j}$ with large intersection with $X$, which contradicts the inductive coloring of $Z_{u,v}^{i,j}$. 

The next proposition elaborates more on the existence of such a $j_0$.

\begin{proposition}\label{prop:dunhappy}
If $d>\frac{N_0}{s(r-1)!}$, then there exists index $1\leq j_0\leq s/2$ such that
\begin{align*}
    \left|d-\frac{mD_{j_0}}{(r-1)!}\right|\geq \frac{N_0}{2s(r-1)!}
\end{align*}
for every $m \in \ZZ$.
\end{proposition}

\begin{proof}
Let $M_0=\frac{N_0}{s(r-1)!}$. Note that by (\ref{eq:5}) we can write
\begin{align*}
    \frac{D_j}{(r-1)!}=(s-j+1)\frac{N_0}{s(r-1)!}=(s-j+1)M_0,
\end{align*}
for every $1\leq j \leq s/2$. Therefore, every number of the form $\frac{mD_j}{(r-1)!}$ for $m\in \ZZ^+$ and $1\leq j\leq s/2$ is a multiple of $M_0$. Moreover, the least non-zero common term among the sequences $\left\{\frac{mD_j}{(r-1)!}\right\}_{m\in \ZZ^+}$ for~$1\leq j \leq s/2$, i.e.,
\begin{align*}
    \min\bigcap_{1\leq j \leq s/2}\left\{\frac{mD_j}{(r-1)!}:\: m\in \ZZ^+\right\}=\min\bigcap_{1\leq j \leq s/2}\left\{m(s-j+1)M_0:\: m\in \ZZ^+\right\}
\end{align*}
is equal to $LM_0$, where $L=\lcm(s/2+1,\ldots,s)$. 

Since every number in $\{1,\ldots,s/2\}$ has a nontrivial multiple inside $\{s/2+1,\ldots,s\}$ we obtain by (\ref{eq:4}) that
\begin{align*}
    L=\lcm(s/2+1,\ldots, s)=\lcm(1,\ldots,s)=e^{s+O(s/\log s)}\geq e^{0.9s},
\end{align*}
for $s=\frac{1}{0.9}\log(1/5\epsilon)\geq \frac{1}{0.9}\log(1/5\epsilon_0)$ and $\epsilon_0$ sufficiently small. Hence, by (\ref{eq:5}) and (\ref{eq:6}) we have
\begin{align*}
    d\leq \frac{wN_0}{2}= \frac{N_0e^{0.9s}}{2s(r-1)!}\leq \frac{LN_0}{2s(r-1)!}=\frac{L}{2}M_0.
\end{align*}

Let $pM_0$ be the multiple of $M_0$ closest to $d$. Since $d>M_0$, we clearly have that $p\neq 0$. By definition,
\begin{align*}
    pM_0=\frac{pN_0}{s(r-1)!}\leq d +\left|\frac{pN_0}{s(r-1)!}-d\right|\leq d+ \frac{M_0}{2} < LM_0.
\end{align*}
Therefore, by the minimality of $LM_0$, there exists an index $1\leq j_0 \leq s/2$ such that $pM_0$ is not a multiple of $\frac{D_{j_0}}{(r-1)!}=(s-j_0+1)M_0$. Since, by the definition of $p$, all the other numbers of the form $mM_0$ have distance at least $\frac{M_0}{2}=\frac{N_0}{2s(r-1)!}$ to $d$, Proposition \ref{prop:dunhappy} follows. 
\end{proof}

We now prove that there exists a set $Y_{i,j}$ with a large proportion of elements of $X$.

\begin{proposition}\label{prop:largeprop}
There exist indices $(i_1,j_1)\in [w]\times [s/2]$ such that $|X\cap Y_{i_1,j_1}|\geq k/s.$
\end{proposition}

\begin{proof}
Let $I\subseteq [w]\times [s/2]$ be set of pair of indices defined by
\begin{align*}
    I=\left\{(i,j)\in [w]\times [s/2]:\:X\cap Y_{i,j}\neq \emptyset\right\},
\end{align*}
and let $\mathcal{Y}=\bigcup_{(i,j)\in I}Y_{i,j}$. By (\ref{eq:5}) and (\ref{eq:6}) we obtain that the difference between two consecutive terms of $X$ is bounded by
\begin{align*}
    |x_{h+1}-x_h|\leq (1+2\epsilon)d \leq (1+2\epsilon)\frac{e^{0.9s}N_0}{2s(r-1)!} < \frac{kN_0}{4s}\leq \frac{k(s-j+1)N_0}{2s^2}=rtD_{j}=|Y_{i,j}|,
\end{align*}
for $k\geq 2^rr!\epsilon^{-1}\log^r(1/5\epsilon)\geq \epsilon^{-1}/(r-1)!$. That is, the difference between two consecutive terms of $X$ is smaller than the size of an interval $Y_{i,j}$ for $(i,j)\in [w]\times [s/2]$. This implies that all intervals in $\mathcal{Y}$ must be consecutive. Recall that by construction two intervals $Y_{i,j}$ and $Y_{i',j'}$ are consecutive if $(i,j)$ and $(i',j')$ are consecutive in the lexicographical ordering of $[w]\times [s/2]$.

If $|I|\leq 2$, then by the pigeonhole principle there exist indices $(i_1,j_1)$ such that $|X\cap Y_{i_1,j_1}|\geq k/2\geq k/s$ for $\epsilon_0$ sufficiently small. Thus we may assume that $|I|>3$. This implies that there exists at least one pair of indices $(i',j')$ such that $Y_{i',j'}$ is neither the first or last interval of $\mathcal{Y}$.

Let $X\cap Y_{i',j'}=\{x_h,\ldots,x_{h+b-1}\}$, where $b=|X\cap Y_{i',j'}|$. Since $Y_{i',j'}$ is not one of intervals in the extreme of $\mathcal{Y}$, we obtain that~$2\leq h\leq h+b-1\leq k-1$ and in particular there exists points $x_{h-1}$ and $x_{h+b}$ outside of $Y_{i',j'}$. Then a simple computation gives us that
\begin{align*}
    |Y_{i',j'}|\leq |x_{h+b}-x_{h-1}|\leq (b+1+2\epsilon)d < 2bd
\end{align*}
and consequently
\begin{align}\label{eq:middle}
    |X\cap Y_{i',j'}|=b > \frac{|Y_{i',j'}|}{2d}
\end{align}
for any $Y_{i',j'}$ not on the extremes of $\mathcal{Y}$.

We split the proof into two cases depending on the size of $d$. If $d\leq \frac{N_0}{s(r-1)!}$, then (\ref{eq:5}) and (\ref{eq:middle}) give that
\begin{align*}
    |X\cap Y_{i',j'}| > \frac{|Y_{i',j'}|}{2d}=\frac{rtD_{j'}}{2d}\geq \frac{k(s-j'+1)(r-1)!}{4s}\geq\frac{k(r-1)!}{8}\geq \frac{k}{s}
\end{align*}
for every $Y_{i',j'}$ not on the extremes and sufficiently large $s$. Taking $(i_1,j_1)$ as one such $(i',j')$ gives the desired result.

Now suppose that $d>\frac{N_0}{s(r-1)!}$. Let $j_0$ be the index provided by Proposition \ref{prop:dunhappy}. In particular, it holds that
\begin{align}\label{eq:7}
    \left|d-\frac{mrD_{j_0}}{q}\right|\geq \frac{N_0}{2s(r-1)!}
\end{align}
for every $m\in \ZZ$ and $1\leq q \leq r$. Suppose that $X\cap Y_{i,j_0}\neq \emptyset$ for some $1\leq i \leq w$. Our goal is to apply Lemma \ref{lem:water} with $D=D_{j_0}$, $\delta=1/4sr!$ to the interval $(\min(Y_{i,j_0})-1,\max(Y_{i,j_0})]=(\beta_{i,j_0},\beta_{i,j_0}+rtD_j]$ labeled with our extension of $\chi_{i,j_0}$. In order to verify the assumptions of the lemma note that 
\begin{align*}
    \frac{N_0}{2s(r-1)!}=\frac{D_{j_0}}{2(s-j_0+1)(r-1)!}\geq \frac{D_{j_0}}{2s(r-1)!}>\delta r D_{j_0}
\end{align*}
and therefore by $(\ref{eq:7})$ we have 
\begin{align*}
    d\notin \bigcup_{m \in \ZZ}\bigcup_{q=1}^r\left(\left(\frac{m}{q}-\delta\right)rD_{j_0},\left(\frac{m}{q}+\delta\right)rD_{j_0}\right).
\end{align*}
Consequently, the conclusion of the lemma gives to us that any arithmetic progression of intervals of radius $\delta r D_{j_0}$ with common difference $d$ and a monochromatic transversal of label~$+1$ inside the interval $(\min(Y_{i,j_0})-1,\max(Y_{i,j_0})]$ has length bounded by~$3r/\delta$. This is true in particular for $\bigcup_{i=0}^{k-1}B(a+id,\epsilon d)$, since by (\ref{eq:5}) and (\ref{eq:6}) we have
\begin{align*}
    \epsilon d\leq \frac{\epsilon w N_0}{2}=\frac{N_0}{10s(r-1)!}= \frac{D_{j_0}}{10(s-j_0+1)(r-1)!}\leq \frac{D_{j_0}}{5s(r-1)!}<\delta r D_{j_0}.
\end{align*}
Hence, because $X$ is transversal of label $+1$ of $\bigcup_{i=0}^{k-1}B(a+id,\epsilon d)$, the conclusion of Lemma \ref{lem:water} gives for $k\geq 2^rr!\epsilon^{-1}\log^r(1/5\epsilon)> \frac{32}{3}r^2\epsilon^{-1}\log(1/5\epsilon)$ that 
\begin{align}\label{eq:special}
|X\cap Y_{i,j_0}|\leq \frac{3r}{\delta}=12sr! r=\frac{40}{3}r!r\log(1/5\epsilon)<\frac{5}{4}\epsilon(r-1)!k.
\end{align}

However, by (\ref{eq:5}), (\ref{eq:6}) and (\ref{eq:middle}) we have
\begin{align}\label{eq:*}
    |X\cap Y_{i',j'}| > \frac{|Y_{i',j'}|}{2d}=\frac{rtD_{j'}}{2d}\geq  \frac{1}{w N_0} \cdot \frac{k(s-j'+1)N_0}{2s^2} \geq \frac{k}{4w s}=\frac{5}{4}\epsilon(r-1)!k
\end{align}
for any $Y_{i',j'}$ in the middle of $\mathcal{Y}$. Comparing (\ref{eq:special}) and (\ref{eq:*}) yields that $|X\cap Y_{i,j_0}|<|X\cap Y_{i',j'}|$ for any interval $Y_{i',j'}$ in the middle of $\mathcal{Y}$. Thus $Y_{i,j_0}$ cannot be a middle interval and we obtain that if $(i,j_0)\in I$, then $Y_{i,j_0}$ is either the first or last interval of $\mathcal{Y}$. Therefore, we can have at most two occurrences of $j_0$ in $I$ and consequently the entire location of $I$ is contained between those two occurrences, i.e., $I \subseteq \{(i,j_0),(i,j_0+1),\ldots,(i+1,j_0-1),(i+1,j_0)\}$ for some $1\leq i \leq w-1$. Hence, the set $I$ has at most $s/2+1$ elements and by the pigeonhole principle there exists a pair of indices $(i_1,j_1)\in I$ such that $|X\cap Y_{i_1,j_1}|\geq k/(s/2+1) \geq k/s$.
\end{proof}

Let $(i_1,j_1)$ be the indices given by Proposition \ref{prop:largeprop}. Next we apply Proposition \ref{prop:simplelowerbound} to the set $Y_{i_1,j_1}$ labeled by $\chi_{i_1,j_1}$ with $D=D_{j_1}$, $\ell=k/s$ and $\epsilon$-approximate arithmetic progression $X\cap Y_{i_1,j_1}$. Note that by (\ref{eq:5}) the hypothesis concerning $r$, $t$ and $\ell$ in the statement holds since
\begin{align*}
t(r+1)+2=\frac{(r+1)k}{2rs}+2<\frac{k}{s}=\ell  
\end{align*}
for $r\geq 2$ and $k\geq 2^rr!\epsilon^{-1}\log^r(1/5\epsilon)\geq 80\log(1/5\epsilon)/9$. Also a $D_j$-block of $Y_{i_1,j_1}$ is an interval of the form $Z_{u,v}^{i_1,j_1}$. Hence, by the conclusion of the proposition, there exists $Z_{u,v}^{i_1,j_1}$ such that~$|X\cap Z_{u,v}^{i_1,j_1}|\geq \ell/(r-1)>k/rs$. Since each set $Z_{u,v}^{i,j}$ was $(r-1)$-colored inductively not to contain an $AP_{k/rs}(\epsilon)$, we reach a contradiction. Thus there is no monochromatic $AP_k(\epsilon)$ in~$[N_1]$. In view of (\ref{eq:5}) we have
\begin{align*}
    N_1=rwt(D_1+\ldots+D_{s/2})&=\frac{ke^{0.9s}N_0}{2s^3(r-1)!}\left(s+\ldots+\left(\frac{s}{2}+1\right)\right)\\
    &\geq \frac{kN_0}{40\epsilon s(r-1)!}\geq \frac{kN_0}{50(r-1)!\epsilon\log(1/\epsilon)}.
\end{align*}
Consequently, in view of $s=O(\log(1/5\epsilon))$ we obtain by induction that
\begin{align*}
    N_1\geq \frac{k}{50(r-1)!\epsilon \log(1/\epsilon)}\cdot\frac{c_r'\left(\frac{k}{rs}\right)^{r-1}}{\epsilon^{r-2}\log(1/\epsilon)^{\binom{r}{2}-1}}\geq c_r\frac{k^r}{\epsilon^{r-1}\log(1/\epsilon)^{\binom{r+1}{2}-1}}.
\end{align*}

\end{proof}

\section{Proof of Theorem \ref{th:approximateszemeredi}}\label{pf:szemeredi}

\subsection{Lower bound}
For positive integers $k$ and $N$, recall that $f(N,1,k)$, sometimes denoted by $r_k(N)$, is defined to be the size of the largest set $A \subseteq [N]$ without an arithmetic progression of length $k$. A classical result of Behrend \cite{B46} shows that,
\begin{align*}
    f(N,1,3)> N\exp(-c\sqrt{\log N}),
\end{align*}
for a positive constant $c$ (see \cite{E11, GW01} for slightly improvements). In \cite{R61} (See also \cite{LL01}) the argument was generalized to yield that
\begin{align}\label{eq:8}
    f(N,1,k) > N\exp\left(-c(\log N)^{1/\ell}\right),
\end{align}
where $\ell=\lceil \log_2 k \rceil$ and $k\geq 3$ and $c$ is a constant depending only on $k$. We will use the last result as a building block for our construction.

Before we turn our attention to the lower bound construction, we will state a preliminary result about $\epsilon$-approximate arithmetic progressions. Given a set of $k$ integers, one can identify them as an $AP_k$ by the common difference between the elements. Unfortunately, the same is not true for an $AP_k(\epsilon)$. On the positive side, the next result shows that if a set of $k$ elements is an $AP_k(\epsilon)$, then the differences of consecutive terms are almost equal.

\begin{proposition}\label{prop:characterization}
Given $0<\epsilon<1/10$, let $X=\left\{x_0,\ldots, x_{k-1}\right\}$ be an $\AP_k(\epsilon)$. Then for every pair of indices $0\leq i, j \leq k-2$ the following holds
\begin{align*}
    \left|\frac{|x_{j+1}-x_j|}{|x_{i+1}-x_i|}-1\right|< 5\epsilon.
\end{align*}
\end{proposition}

\begin{proof}
Since $X$ is an $\AP_k(\epsilon)$, there exist $a$ and $d$ such that $|x_i-(a+id)|<\epsilon d$ for~$0\leq i \leq k-1$. Therefore, a simple computation shows that
\begin{align*}
    1-5\epsilon<\frac{(1-2\epsilon)d}{(1+2\epsilon)d}<\frac{|x_{j+1}-x_j|}{|x_{i+1}-x_i|}<\frac{(1+2\epsilon)d}{(1-2\epsilon)d}<1+5\epsilon
\end{align*}
for $0<\epsilon<1/10$ and $0\leq i,j \leq k-2$.
\end{proof}

We now prove the lower bound of Theorem \ref{th:approximatewaerden} for one dimension. 

\begin{lemma}\label{lem:lowerszemeredi}
Let $k\geq 3$ and $0<\epsilon\leq 1/125$. Then there exists a positive constant~$c_1$ depending only on $k$ and an integer $N_0:=N_0(k,\epsilon)$ such that the following holds. If $N\geq N_0$, then there exists a set $A\subseteq [N]$ without $AP_k(\epsilon)$ such that
\begin{align*}
    |A|\geq N^{1-c_1(\log(1/\epsilon))^{\frac{1}\ell-1}}
\end{align*}
for $\ell=\lceil \log_2 k\rceil$.
\end{lemma}

\begin{proof}
For integers $a,b$, let $S_k([a,b])$ be the largest subset in the interval $[a,b]$ without any arithmetic progression $\AP_k$ of length $k$. By a simple translation, one can note that $S_k([a,b])$ has the same size as $S_k([b-a+1])$ and by (\ref{eq:8}) we have
\begin{align}\label{eq:9}
    |S_k([a,b])|=f(b-a+1,1,k)\geq (b-a+1)\exp\left(-c(\log(b-a+1))^{1/\ell}\right),
\end{align}
for a positive constant $c$ and $\ell=\lceil \log_2 k \rceil$.

Let $q=\frac{1}{25\epsilon}\geq 5$ be an integer and $h$ be largest exponent such that $q^h\leq N < q^{h+1}$. For such a choice of $q$ and $h$, we construct the set
\begin{align*}
    A=\left\{s\in [N]:\: a=s_0+s_1q+\ldots+s_{h-1}q^{h-1}\right\},
\end{align*}
where $s_{h-1}\in S_k([0,q-1])$ and $s_i \in S_k\left(\left[2q/5,3q/5\right]\right)$ for $0\leq i \leq h-2$. Our goal is to show that~$A$ satisfies the conclusion of Lemma \ref{lem:lowerszemeredi}.

First note by (\ref{eq:9}) that
\begin{align*}
    |A|&= |S_k([0,q-1])|\cdot \left|S_k\left(\left[2q/5,3q/5\right]\right)\right|^{h-1}\\
    &\geq \frac{q}{\exp(c(\log q)^{1/\ell})}\cdot \left(\frac{q}{5\exp(c(\log q/5)^{1/\ell})}\right)^{h-1}\\
    &\geq \frac{q^h}{\exp(c(\log q)^{1/\ell})(5\exp(c(\log q)^{1/\ell}))^{h-1}}\\
    &\geq \frac{N}{5^{h-1}q\exp(c(\log q)^{1/\ell})^h}\geq \frac{N}{q\exp(c'h(\log q)^{1/\ell})},
\end{align*}
and in view of $h\leq \frac{\log N}{\log q}$ and our choice of $q$ we obtain that
\begin{align*}
    |A|\geq \frac{20\epsilon N}{\exp\left(c'\log N(\log q)^{1/\ell-1}\right)} \geq N^{1-c_1\log (1/\epsilon)^{1/\ell-1}}
\end{align*}
for sufficiently large $N$ and appropriate constant $c_1$ depending only on $k$. Therefore the set~$A$ has the desired size. It remains to prove that $A$ is $\AP_k(\epsilon)$-free. 

Suppose that there exists an $\epsilon$-approximate arithmetic progression $X=\{x_0,\ldots,x_{k-1}\}$ in $A$. For each $0\leq i \leq k-1$, write $x_i=\sum_{j=0}^{h-1}x_{i,j}q^j$. Since all $x_i$'s are distinct, there exists a maximal index $j_0$ such that the elements of $X_{j_0}=\{x_{i,j_0}\}_{0\leq i \leq k-1}$ are not all equal. By construction of $A$ the set $X_{j_0}$ fails to be an $\AP_k$. Therefore there exists two indices $0\leq i_1, i_2 \leq k-2$ such that
\begin{align}\label{eq:10}
    |x_{i_1+1,j_0}-x_{i_1,j_0}|\neq |x_{i_2+1,j_0}-x_{i_2,j_0}|.
\end{align}

For $0\leq i \leq k-1$, note that 
\begin{align*}
    |x_{i+1}-x_i|=\left|\sum_{j=0}^{h-1}(x_{i+1,j}-x_{i,j})q^j\right|=\left|\sum_{j=0}^{j_0}(x_{i+1,j}-x_{i,j})q^j\right|
\end{align*}
by the maximality of $j_0$. Thus by the triange inequality we obtain that
\begin{align}\label{eq:11}
    \left||x_{i+1}-x_i|-|x_{i+1,j_0}-x_{i,j_0}|q^{j_0}\right|\leq \sum_{j=0}^{j_0-1}|x_{i+1,j}-x_{i,j}|q^j.
\end{align}
Moreover, recalling that $x_{i,j} \in \left[2q/5,3q/5\right]$ for $0\leq j \leq h-2$ we infer that
\begin{align*}
    \sum_{j=0}^{j_0-1}|x_{i+1,j}-x_{i,j}|q^j\leq \sum_{j=0}^{j_0-1}\frac{q^{j+1}}{5} \leq \frac{2q^{j_0}}{5} 
\end{align*}
for $q\geq 2$. The last inequality combined with (\ref{eq:11}) gives us that
\begin{align}\label{eq:12}
  \left||x_{i+1}-x_i|-|x_{i+1,j_0}-x_{i,j_0}|q^{j_0}\right|\leq \frac{2}{5}q^{j_0},
\end{align}
for $0\leq i \leq k-2$. Hence by (\ref{eq:10}) we have that
\begin{align}\label{eq:13}
    \big||x_{i_2+1}-x_{i_2}|-|x_{i_1+1}-x_{i_1}|\big|&\geq \big||x_{i_2+1,j_0}-x_{i_2,j_0}|-|x_{i_1+1,j_0}-x_{i_1,j_0}|\big|q^{j_0}-\frac{4}{5}q^{j_0}\nonumber\\
    &\geq q^{j_0}-\frac{4}{5}q^{j_0}=\frac{q^{j_0}}{5}
\end{align}

On the other hand, Proposition \ref{prop:characterization} for $i_1$ and $i_2$ together with (\ref{eq:12}) gives us that
\begin{align*}
    \big||x_{i_2+1}-x_{i_2}|-|x_{i_1+1}-x_{i_1}|\big| <5\epsilon|x_{i_1+1}-x_{i_1}|<5\epsilon\left(|x_{i_1+1,j_0}-x_{i_1,j_0}|+\frac{2}{5}\right)q^{j_0}.
\end{align*}
Since $x_{i,j_0} \in [0,q-1]$ for every $0\leq i \leq k-1$ and $\epsilon q=1/25$ we have
\begin{align*}
     \big||x_{i_2+1}-x_{i_2}|-|x_{i_1+1}-x_{i_1}|\big|<5\epsilon q^{j_0+1}=\frac{q^{j_0}}{5},
\end{align*}
which contradicts (\ref{eq:13}).
\end{proof}

For higher dimensions the result follow as a corollary of Lemma \ref{lem:lowerszemeredi}. Recall by Definition~\ref{def:cube} that an $\epsilon$-approximate cube $C_\epsilon(m,k)$ is just an multidimensional version of an $\AP_k(\epsilon)$ 

\begin{corollary}
Let $k\geq 3$ and $m\geq 1$ be integers and $0<\epsilon\leq 1/125$. Then there exists an integer $N_0:=N_0(k,\epsilon)$ and a positive constant $c_1$ depending on $k$ such that the following holds. If $N\geq N_0$, then there exists a set $S\subseteq [N]^m$ without $C_{\epsilon}(m,k)$ such that
\begin{align*}
    |S|\geq N^{m-c(\log(1/\epsilon))^{\frac{1}\ell-1}}
\end{align*}
for $\ell=\lceil \log_2(k-1)\rceil$.
\end{corollary}

\begin{proof}
Let $N_0$ be the integer given by Lemma \ref{lem:lowerszemeredi} and let $A\subseteq [N]$ be the set such that $A$ has no~$\AP_k(\epsilon)$ for $N\geq N_0$. Set $S=A\times [N]^{m-1}$, i.e., $S=\left\{(s_1,\ldots,s_m):\: s_1\in A,\, s_2,\ldots,s_m \in [N]\right\}$. Note that $S$ has the desired size since
\begin{align*}
    |S|=N^{m-1}|A|\geq N^{m-c(\log(1/\epsilon))^{\frac{1}\ell-1}}.
\end{align*}
We claim that $S$ is free of $C_{\epsilon}(m,k)$.

Suppose that the claim is not true and let $X=\{x_{\vec{v}}:\: \vec{v}\in \{0,\ldots,k-1\}^m\}$ be an $C_{\epsilon}(m,k)$ in $S$. By definition there exists $\vec{a} \in \RR^m$ and $d>0$ such that $||x_{\vec{v}}-(\vec{a}+d\vec{v})||<\epsilon d$ for every~$\vec{v} \in \{0,\ldots,k-1\}^m$. In particular, when applied to $\{te_1=(t,0,\ldots,0):\: 0\leq t \leq k-1\}$ the observation gives us that
\begin{align*}
    |x_{te_1,1}-(a_1+td)| \leq \left((x_{te_1,1}-(a_1+dt))^2+\sum_{i=2}^m(x_{te_1,i}-a_i)^2\right)^{1/2}=||x_{te_1}-(\vec{a}+dte_1)||<\epsilon d
\end{align*}
Therefore, the set $\{x_{te_1,1}\}_{0\leq t \leq k-1}\subseteq A$ is an $\AP_k(\epsilon)$, which contradicts our choice of $A$.
\end{proof}

\subsection{Upper bound}

As in the upper bound of $W_\epsilon(k,r)$, our proof of the upper bound of $f_\epsilon(N,m,k)$ will use an iterative blow-up construction. It is worth to point out that a similar proof was obtained independently by Dumitrescu in \cite{D11}. While both proofs use a blow-up construction, the author of \cite{D11} finishes the proof with a packing argument. Here we will follow the approach of \cite{HKSS19, HKSS19-2}, which uses an iterative blow-up construction combined with an average argument to estimate the largest subset of a grid without a class of configurations of a given size. This approach allows us to slightly improve the constants in the result.


The proof is split into two auxiliary lemmas.

\begin{lemma}\label{lem:upperszemeredi}
Given positive real numbers $\alpha$, $\epsilon>0$ and integers $m\geq 1$ and $k\geq 3$, there exists~$N_0:=N_0(\alpha,\epsilon,m,k)\leq \left(k\sqrt{m}/\epsilon\right)^{2k^m\log(1/\alpha)}$ and a subset $A \subseteq [N]^m$ with the property that any $X\subseteq A$, $|X|\geq \alpha|A|$ contains a $C_{\epsilon}(m,k)$.
\end{lemma}

\begin{proof}
For $m$ and $k$, let $\Delta$ be the standard cube $C(m,k)$ of dimension $m$ over~$\{0,\ldots,k-1\}$, i.e., $\Delta$ is the set of all $m$-tuples $\vec{v}=\{v_1,\ldots,v_m\}\in \{0,\ldots,k-1\}^m$. Viewing $\Delta$ as an $m$-dimensional lattice in the Euclidean space, we note that $\diam(\Delta)=(k-1)\sqrt{m}$, while the minimum distance between two vertices in $\Delta$ is one.

Similarly as in the proof of the upper bound of Theorem \ref{th:approximatewaerden}, we consider an iterated blow-up of the cube. For integers $r$ and $t=k\sqrt{m}/\epsilon$, let $A_r$ be the following $r$-iterated blow-up of a cube
\begin{align*}
    A_r=\left\{ \vec{v}_0+t\vec{v}_1+\ldots+t^{r-1}\vec{v}_{r-1}:\: \vec{v}_0,\ldots,\vec{v}_{r-1}\in \Delta,\, t=\frac{k\sqrt{m}}{\epsilon}\right\}.
\end{align*}
Alternatively, we can view $A_r$ as the product $\prod_{i=1}^mB_r^{(i)}$ of $m$ identical copies of 
\begin{align*}
 B_r=\left\{b_0+tb_1+\ldots+t^{r-1}b_{r-1}:\: (b_0,\ldots,b_{r-1})\in \{0,1\ldots,k-1\}^r,\, t=\frac{k\sqrt{m}}{\epsilon} \right\},
\end{align*} an $r$-iterated blow-up of the standard $\AP_k$. Note by the construction that $|A_r|=k^{rm}$. The next proposition shows that fixed $\alpha>0$, for a sufficiently large $r$ any $\alpha$-proportion of $A_r$ will contain a $C_{\epsilon}(m,k)$.

\begin{proposition}\label{prop:simpleupper}
Let $0<\alpha<1$ be a real number and $r$ a positive integer such that $\alpha> \left(\frac{k^m-1}{k^m}\right)^r$. Then every $X\subseteq A_r$ with $|X|\geq \alpha|A_r|$ contains a $C_{\epsilon}(m,k)$.  
\end{proposition}

\begin{proof}
The proof is by induction on $r$. If $r=1$, then $A_1=\Delta$ and $\alpha>\frac{k^m-1}{k^m}$. Let $X\subseteq A_1$ with~$|X|\geq \alpha |A_1|$. Thus
\begin{align*}
    |X|\geq \alpha |A_1| >\frac{k^m-1}{k^m}\cdot k^m=k^m-1,
\end{align*}
which implies that $X=\Delta$. So $X$ contains a cube $C(k,m)$ and in particular an~$\epsilon$-approximate cube.

Now suppose that the proposition is true for $r-1$ and we want to prove it for $r$. First, we partition $A_r$ into $\bigcup_{\vec{u}\in \Delta} A_{r,\vec{u}}$, where
\begin{align*}
    A_{r,\vec{u}}=\left\{\vec{v}_0+t\vec{v}_1+\ldots+t^{r-2}\vec{v}_{r-2}+t^{r-1}\vec{u}:\:\vec{v_0},\ldots,\vec{v}_{r-2}\in \Delta,\, t=\frac{k\sqrt{m}}{\epsilon}\right\}.
\end{align*}
Note that by definition $A_{r,\vec{u}}$ is a translation of $A_{r-1}$ by $t^{r-1}\vec{u}$. In particular, this implies that~$|A_{r,\vec{v}}|=k^{(r-1)m}$. Let $X\subseteq A_r$ with $|X|\geq \alpha |A_r|$ be given. We will distinguish two cases:

\underline{Case 1:} $X\cap A_{r,\vec{u}}\neq \emptyset$ for all $\vec{u}\in \Delta$.

For each $\vec{u}\in \Delta$ choose an arbitrary vector $w(\vec{u}) \in X\cap A_{r,\vec{u}}$. We will observe that $\{w(\vec{u})\}_{\vec{u}\in \Delta}$ forms a $C_{\epsilon}(m,k)$. To testify that, set $\vec{a}=(0,\ldots,0)$ and $d=t^{r-1}$. Write $w(\vec{u})=\sum_{i=0}^{r-2}t^i\vec{w}_i+t^{r-1}\vec{u}$ with $\vec{w}_i \in \Delta$. Thus, a computation shows that
\begin{align*}
    ||w(\vec{u})-(\vec{a}+d\vec{u})||=||w(\vec{u})-t^{r-1}\vec{u}||=\left|\left|\sum_{i=0}^{r-2}t^i\vec{w}_i\right|\right|\leq \sum_{i=0}^{r-2}t^i||\vec{w}_i||
\end{align*}
for $\vec{w}_0,\ldots,\vec{w}_{r-2} \in \Delta$. Since $\diam(\Delta)=(k-1)\sqrt{m}$, it follows that
\begin{align*}
    ||w(\vec{u})-(\vec{a}+d\vec{u})||\leq (k-1)\sqrt{m}\left(\sum_{i=0}^{r-2}t^i\right)\leq kt^{r-2}\sqrt{m}<\epsilon t^{r-1}=\epsilon d,
\end{align*}
by our choice of $t$. Since $\{w(\vec{u})\}_{\vec{u}\in \Delta}\subseteq X$, we conclude that $X$ contains an $C_{\epsilon}(m,k)$.

\underline{Case 2:} There exists $\vec{u_0} \in \Delta$ with $X\cap A_{r,\vec{u}_0}=\emptyset$.

Since $|X|\geq \alpha |A_r|$ and $|\Delta|=k^m$, by an average argument there exists $\vec{u}_1\in \Delta$ such that
\begin{align*}
    |X\cap A_{r,\vec{u}_1}|\geq \frac{\alpha|A_r|}{k^m-1}=\frac{\alpha k^m|A_{r-1}|}{k^m-1}.
\end{align*}
Set $X'=X\cap A_{r,\vec{u}_1}$ and $\alpha'=\frac{\alpha k^m}{k^m-1}$. Note that
\begin{align*}
    \alpha'=\frac{\alpha k^m}{k^m-1}>\left(\frac{k^m-1}{k^m} \right)^r\cdot \frac{k^m}{k^m-1}=\left(\frac{k^m-1}{k^m} \right)^{r-1}.
\end{align*}
Therefore, viewing $A_{r,\vec{u}}$ as a copy of $A_{r-1}$ by the induction assumption we obtain that $X'\subseteq X$ contains an $C_{\epsilon}(m,k)$. 
\end{proof}

Let $r$ be the smallest integer such that $\left(\frac{k^m-1}{k^m}\right)^r< \alpha$ and set $A=A_r$. A computation shows that 
\begin{align*}
    r=\left\lceil \frac{\log(1/\alpha)}{\log \frac{k^m}{k^m-1}} \right\rceil < 2k^m\log(1/\alpha).
\end{align*}
Therefore by Proposition \ref{prop:simpleupper} we have that any set $X\subseteq A$ with $|X|\geq \alpha|A|$ contains an $C_{\epsilon}(m,k)$. Finally, by the construction of $A$ we have that $A\subseteq [N_0]^m$ for
\begin{align*}
    N_0\leq \diam(B_r)+1 =(k-1)(1+t+\ldots+t^{r-1})+1 \leq kt^{r-1}\leq \left(\frac{k\sqrt{m}}{\epsilon}\right)^{2k^m\log(1/\alpha)}.
\end{align*}
\end{proof}

Lemma \ref{lem:upperszemeredi} gives us a set $A\subseteq [N]^m$ such that any $\alpha$-proportion contains a $C_{\epsilon}(m,k)$. However, this is still not good enough, since to obtain an upper bound we need a similar result for $[N]^m$. The next lemma shows by an average argument that the property of $A$ can be extended to $[N]^m$ by losing a factor of a power of two in the proportion $\alpha$.

\begin{lemma}\label{lem:average}
Let $A \subseteq [N]^m$ be a configuration in the grid. For any $X\subseteq [N]^m$ with $|X|\geq \alpha N^m$, there exists a translation $A'$ of $A$ such that $|X\cap A'|\geq \frac{\alpha}{2^m} |A'|$.
\end{lemma}

\begin{proof}
Consider a random translation $A'=A+\vec{u}$, where $\vec{u}=(u_1,\ldots,u_m)$ is an integer vector chosen uniformly inside $[-N+1,N]^m$. For every vector $\vec{x} \in X$, there exists exactly $|A|$ elements~$\vec{v} \in [-N+1,N]^m$ such that $\vec{x}-\vec{v}\in A$. This means that $\PP(\vec{x} \in A')=\PP(\vec{x}-\vec{u}\in A)=\frac{|A|}{(2N)^m}$. Therefore
\begin{align*}
    \EE(|X\cap A'|)=\sum_{\vec{x} \in X}\PP(\vec{x} \in A')=\frac{|X||A|}{(2N)^m}\geq \frac{\alpha}{2^m}|A|=\frac{\alpha}{2^m}|A'|
\end{align*}
Consequently, by the first moment method, there is $\vec{u}$ and $A'$ satisfying our conclusion.
\end{proof}

We finish the section putting everything together.

\begin{proposition}
Let $N$, $m$ and $k$ be integers and $\epsilon>0$. Then there exists a positive constant $c_2$ depending only on $k$ and $m$ such that the following holds. If $S\subseteq [N]^m$ is such that
\begin{align*}
    |S|> N^{m-c_2(\log(1/\epsilon))^{-1}},
\end{align*}
then $S$ contains an $C_{\epsilon}(m,k)$.
\end{proposition}

\begin{proof}
Set $\alpha_0=2^mN^{-c'(\log(1/\epsilon))^{-1}}$ where $c'=(4k^m\log(k\sqrt{m}))^{-1}$. Let $N_0=N_0(\alpha_0/2^m,\epsilon,m,k)$ be the integer obtained by Lemma \ref{lem:upperszemeredi} and $A\subseteq [N_0]$ be the set such that any $X\subseteq A$ with~$|X|\geq \frac{\alpha_0}{2^m}|A|$ contains an $C_{\epsilon}(m,k)$. Note that
\begin{align*}
    N_0&\leq \left(\frac{k\sqrt{m}}{\epsilon}\right)^{2k^m\log(2^m/\alpha_0)} = \exp\left(\frac{2c'k^m\log N \log(k\sqrt{m}/\epsilon)}{\log(1/\epsilon)} \right)\\
    &\leq \exp\left(4c'k^m\log N \log(k\sqrt{m}) \right)=N,
\end{align*}
which implies that $A\subseteq [N]$.

Let $S \subseteq [N]$ with $|S|\geq \alpha_0 N^m$. Then by Lemma \ref{lem:average}, there exists a translation $A'$ of $A$ such that $|S\cap A'|\geq \frac{\alpha_0}{2^m}|A'|$. Hence, by Lemma \ref{lem:upperszemeredi}, the set $S$ contains a $C_{\epsilon}(m,k)$. The result now follows since
\begin{align*}
    |S|\geq \alpha_0 N^m=2^mN^{m-c'(\log(1/\epsilon))^{-1}}>N^{m-c_2(\log(1/\epsilon))^{-1}}
\end{align*}
for appropriate $c_2$.
\end{proof}

\bibliography{literature}
\end{document}